\documentclass[a4paper,oneside,11pt, reqno]{amsproc}
\usepackage{graphicx}
\usepackage{float}
\usepackage{pstricks}
\usepackage{amscd} 
\usepackage{amsmath}
\usepackage{amsxtra}
\usepackage{tikz}
\usepackage{tikz-cd}
\usepackage[T1]{fontenc}
\usepackage[utf8]{inputenc}
\usepackage{lipsum}
\usepackage{tikz}
\usetikzlibrary{arrows}
\usetikzlibrary{calc}

\usepackage{amsfonts,amssymb,amsmath,amsthm}
\usepackage{url}
\usepackage{enumerate}

\theoremstyle{plain}
\newtheorem{theorem}{Theorem}[section]

\newtheorem{lemma}[theorem]{Lemma}
\newtheorem{proposition}[theorem]{Proposition}
\newtheorem{corollary}[theorem]{Corollary}

\theoremstyle{definition}
\newtheorem{definition}[theorem]{Definition}

\newtheorem{example}[theorem]{Example}

\theoremstyle{remark}
\newtheorem{remark}[theorem]{Remark}

\numberwithin{equation}{section}

\newlength{\struh}
\newlength{\textminustop}
\newcommand{\tria}{\triangle_{n}}
\newcommand{\poly}{\mathbb D \times \mathbb D^{n-1}_*}

\newcommand*{\Ge}{\geqslant}

\newcommand*{\Le}{\leqslant}

\usepackage{mathrsfs}
\usepackage{epsfig}
\usepackage[active]{srcltx}

\newcommand{\ncom}{\newcommand}
\ncom{\bq}{\begin{equation}}
\ncom{\eq}{\end{equation}}
\ncom{\beqn}{\begin{eqnarray*}}
\ncom{\eeqn}{\end{eqnarray*}}
\ncom{\beq}{\begin{eqnarray}}
\ncom{\eeq}{\end{eqnarray}}
\ncom{\nno}{\nonumber}
\ncom{\rar}{\rightarrow}
\ncom{\Rar}{\Rightarrow}
\ncom{\noin}{\noindent}
\ncom{\bc}{\begin{centre}}
\ncom{\ec}{\end{centre}}
\ncom{\sz}{\scriptsize}
\ncom{\rf}{\ref}
\ncom{\sgm}{\sigma}
\ncom{\Sgm}{\Sigma}
\ncom{\dt}{\delta}
\ncom{\Dt}{Delta}
\ncom{\lmd}{\lambda}
\ncom{\Lmd}{\Lambda}
\ncom{\eps}{\epsilon}
\ncom{\pcc}{\stackrel{P}{>}}
\ncom{\dist}{{\rm\,dist}}
\usepackage{cite}
\ncom{\sspan}{{\rm\,span}}
\ncom{\im}{{\rm Im\,}}
\ncom{\sgn}{{\rm sgn\,}}
\ncom{\ba}{\begin{array}}
\ncom{\ea}{\end{array}}
\ncom{\eop}{\hfill{{\rule{2.5mm}{2.5mm}}}}
\ncom{\eoe}{\hfill{{\rule{1.5mm}{1.5mm}}}}
\ncom{\eof}{\hfill{{\rule{1.5mm}{1.5mm}}}}
\ncom{\hone}{\mbox{\hspace{1em}}}
\ncom{\htwo}{\mbox{\hspace{2em}}}
\ncom{\hthree}{\mbox{\hspace{3em}}}
\ncom{\hfour}{\mbox{\hspace{4em}}}
\ncom{\hsev}{\mbox{\hspace{7em}}}
\ncom{\vone}{\vskip 2ex}
\ncom{\vtwo}{\vskip 4ex}
\ncom{\vonee}{\vskip 1.5ex}
\ncom{\vthree}{\vskip 6ex}
\ncom{\vfour}{\vspace*{8ex}}
\ncom{\norm}{\|\;\;\|}
\ncom{\integ}[4]{\int_{#1}^{#2}\,{#3}\,d{#4}}
\ncom{\inp}[2]{\langle{#1},\,{#2} \rangle}
\ncom{\Inp}[2]{\Langle{#1},\,{#2} \Langle}
\ncom{\vspan}[1]{{{\rm\,span}\#1 \}}}
\ncom{\dm}[1]{\displaystyle {#1}}



\thanks{The second author was supported by NBHM Grant 0204-9-2022-R$\&$D-II-5885.}
\begin{document}
\title[Toeplitz operators on the $n$-dimensional Hartogs triangle]{Toeplitz operators on the $n$-dimensional \\ Hartogs triangle}

\author[S. Jain]{Shubham Jain} 
\author[P. Pramanick]{Paramita Pramanick}

\address{Department of Mathematics and Statistics\\
Indian Institute of Technology Kanpur, India}

\email{shubjain@iitk.ac.in, shubhamjain4343@gmail.com}
\email{paramitapramanick@gmail.com}

\begin{abstract} We formally introduce and study Toeplitz operators on the Hardy space of the $n$-dimensional Hartogs triangle. We find a precise relation between these operators  and the Toeplitz operators on the Hardy space of the unit polydisc $\mathbb D^n.$  As an application, we deduce several properties of these operators from their polydisc counterparts. Furthermore, we show that certain results achieved for the Toeplitz operators on the $n$-dimensional Hartogs triangle are not the same as those in the polydisc case.
\end{abstract}

\maketitle

\section{Introduction}
Let $\mathbb N, \mathbb Z_+, \mathbb Z$ and $\mathbb R$ denote the set of positive integers, nonnegative integers, integers, and real numbers, respectively. Let $\mathbb C$ and $\mathbb C_*$ denote the set of complex numbers and nonzero complex numbers, respectively. The symbols $\mathbb T$ and $\mathbb D$ represent the unit circle and unit disc in the complex plane $\mathbb C$. We use the notation $\mathbb D_*$ for $\mathbb D \setminus \{0\}.$ For $n \in \mathbb N$ and a nonempty set $A,$ let $A^n$ denote the $n$-fold Cartesian product of $A$ with itself. For $z =(z_1, \ldots, z_n) \in \mathbb C^n$ and $\alpha = (\alpha_1, \ldots, \alpha_n) \in \mathbb Z^n_+,$ the expression $z^\alpha$ represents the complex number $\prod_{j=1}^n z^{\alpha_j}_j.$ Let $\varepsilon_j$ denote the $n$-tuple in $\mathbb C^n$ with $j$-th entry equal to $1$ and zeros elsewhere.

Let $n \Ge 2$ be a positive integer. The domain $n$-dimensional Hartogs triangle, denoted by $\tria,$ is defined as 
\beqn
\tria :=  \{z \in \mathbb C^n : |z_{1}|  < |z_2| < \ldots < |z_n| < 1\}.
\eeqn
Note that $\tria$ is a bounded and pseudoconvex domain. However, it is neither monomially convex nor polynomially convex.
Consider the biholomorphism $\varphi$ from $\tria$ onto $\poly$ given by the map \beq \label{phi-map} \varphi(z)=\left(\frac{z_1}{z_2}, \ldots, \frac{z_{n-1}}{z_n}, z_n \right), \quad z=(z_1, \ldots, z_n) \in \tria.  \eeq
Note that the Jacobian $J_{\varphi}$ of $\varphi$ is given by
\beq \label{Jacobian phi} J_{\varphi}(z)=\frac{1}{z_2\cdots z_n}, \quad z=(z_1, \ldots, z_n) \in \tria. \eeq
Further, $\varphi^{-1}$ is given by
\allowdisplaybreaks
\beq \label{phi inverse}
\varphi^{-1}(z) = \Big(\prod_{j=1}^n z_j,  \ldots, \prod_{j=n-1}^n  z_j, z_n\Big), \quad z =(z_1, \ldots, z_n) \in \poly. 
\eeq
It is easy to see that the Jacobian $J_{_{\varphi^{-1}}}$ of $\varphi^{-1}$ is given by
\beqn
J_{_{\varphi^{-1}}}(z) = \prod_{j=2}^n z^{j-1}_j, \quad z=(z_1, \ldots, z_n) \in \poly.
\eeqn 
For a bounded domain in $\mathbb C^n,$ let $\mathcal{O}(\Omega)$ denote the space of holomorphic functions on $\Omega$ and let $A(\Omega)$ denote the algebra of holomorphic functions on $\Omega$ that are continuous on $\overline{\Omega}.$ Let $H^{\infty}(\Omega)$ denote the space of bounded analytic functions on $\Omega.$ For a bounded linear operator $T$ on a separable complex Hilbert space $H,$ let $\operatorname{ran}(T)$ and $\operatorname{ker}(T)$ denote the range and kernel of $T,$ respectively.

Recall from \cite[Section 7]{CJP2023} (for $n=2,$ see \cite[Section~3]{M2021}) that the Hardy space $H^2(\tria)$ on $\tria$ is a reproducing kernel Hilbert space of holomorphic functions on $\tria$ endowed with the norm
\begin{align} \label{Norm formula-1}
\|f\|^2=  \sup_{\underset{r_j \in (0, 1)}{j=1, \ldots, n}} \int_{[0, 2\pi]^n}  \!\Big|f\Big(\prod_{j=1}^n r_j e^{i \theta_1}, \prod_{j=2}^n r_j e^{i \theta_2}, \ldots, r_n e^{i \theta_n}\Big)\Big|^2\prod_{j=1}^n r^{2j-1}_j \frac{d\theta}{(2\pi)^n}, 
\end{align}
where $d\theta$ denotes the Lebesgue measure on $[0, 2\pi]^n$ and $f \in H^2(\tria).$
Also, recall that the Hardy space $H^2(\mathbb D^n)$ on the unit polydisc $\mathbb D^n$ is defined as 
\beqn
H^2(\mathbb D^n):=\left\{f \in \mathcal{O}(\mathbb D^n):\sup_{t \in [0,1)}\int_{\mathbb T^n}|f(tz_1, \ldots, tz_n)|^2dm(z) < \infty \right \},
\eeqn
where $dm(z)$ is the normalized Lebesgue measure on $\mathbb T^n.$ 
\begin{remark} \label{Hardy space polydisc}
    The space $H^2(\mathbb D^n)$ may be identified with the closed subspace $H^2(\mathbb T^n)$ of $L^2(\mathbb T^n)$ (see \cite[Theorem~3.4.3]{R1969}). \hfill $\diamondsuit$
\end{remark}
Here is the outline of the paper. In Section~\ref{2}, we provide an alternative expression for the norm on $H^2(\tria),$ see Proposition \ref{norm2}. Further, the Hardy space $H^2(\tria)$ is identified with a closed subspace of $L^2(\mathbb T^n),$ see Proposition \ref{Boundary values identification}. As a consequence, we define inner functions on $\tria,$ see Definition \ref{def-inner}. In Section~\ref{3}, we define the Toeplitz operator $T_{\phi, \triangle_n}$ on $H^2(\tria)$ and study some of its basic properties. In section \ref{4}, we present the main theorem of this paper, see Theorem \ref{Toeplitz- relation}, which provides a concrete relation between $T_{\phi, \triangle_n}$  and the Toeplitz operators on $H^2(\mathbb D^n).$ In the final section, we obtain several applications of the
Theorem \ref{Toeplitz- relation}. These include a characterization of Toeplitz operators defined on the Hardy space $H^2(\tria)$ similar to the familiar Brown-Halmos criterion for Toeplitz operators defined on the Hardy space $H^2(\mathbb D),$ see Theorem \ref{Brown-Halmos}. We also provide a necessary and sufficient condition for $T_{\phi, \triangle_n}$ to be a partial isometry, see Theorem \ref{Partial isometry result}. Furthermore, we discuss analytic Toeplitz operators on $H^2(\tria).$

\section{Hardy space and boundary values} \label{2}
In this section, we establish a correspondence between the Hardy space $H^2(\tria)$ and a closed subspace of $L^2(\mathbb T^n)$ (see Proposition~\ref{Boundary values identification}; for $n=2,$ see \cite[p.~15]{M2021}). This naturally leads us to the notion of the inner function of a member of $H^2(\tria)$ (see Definition~\ref{def-inner}). 

Recall that 
 the {\it distinguished boundary} $\partial_{d}(\Omega)$ of a domain $\Omega$ is the smallest closed subset of the boundary of $\Omega$ such that for any $f \in A(\Omega),$
 \beqn \displaystyle |f(z)| \Le \sup_{w \in \partial_{d}(\Omega)}{|f(w)|}, \quad z \in \Omega.
 \eeqn

We begin with a useful relationship between the distinguished boundary of a domain $D\subseteq \Omega$ with that of the distinguished boundary of the domain $\Omega$. We include a proof for the sake of completeness. 
 
\begin{lemma} \label{silov boundary}
Let $\Omega$ be a bounded domain in $\mathbb{C}^n$ with a distinguished boundary $\partial_{d}(\Omega)$ and let $D$ be a domain contained in $\Omega$ such that $\partial_{d}(\Omega)$ is a subset of the boundary of $D.$ If there exists a continuous map $\Phi: \overline{\Omega} \rightarrow \overline{D}$ such that $\Phi$ is holomorphic on $\Omega$ and satisfies $\Phi(\Omega) = D$ and $\Phi(\partial_{d}(\Omega)) = \partial_{d}(\Omega),$ then the distinguished boundary $\partial_{d}(D)$ of $D$ is equal to $\partial_{d}(\Omega).$
\end{lemma}
\begin{proof}
    Let $f \in A(D).$ Then $f \circ \Phi \in A(\Omega).$ Since $\partial_{d}(\Omega)$ is the distinguished boundary of $\Omega,$ we have 
    \beqn
     |f \circ \Phi(z)| \Le \sup_{w \in \partial_{d}(\Omega)}|f \circ \Phi(w)|, \quad z \in \Omega.
    \eeqn
Consequently, we have $ \displaystyle |f(z)| \Le \sup_{w \in \partial_{d}(\Omega)}|f(w)|$ for any $z \in D.$
    
On the other hand, for $\eta \in \partial_{d}(\Omega),$ there exists a function $f_{\eta} \in A(\Omega)$ such that $f_{\eta}(\eta)=1$ and $|f_{\eta}(z)|<1$ for all $z \in \overline \Omega \setminus\{\eta\}.$ Since $f_{\eta}|_{D}$ is in $A(D),$ this completes the proof.\end{proof}
\begin{remark} \label{Distinguishd boundary same}
 The distinguished boundary of $\tria$ is $\mathbb{T}^n.$ To see this, note that $\tria \subset \mathbb D^n$ and the distinguished boundary $\mathbb T^n$ of $\mathbb D^n$ is a subset of $\overline{\tria}.$ Also, note that $\varphi^{-1}$ extends to be a continuous map from $\overline{\mathbb D}^n$ to $\overline{\tria}$ such that $\varphi^{-1}$ is holomorphic on $\mathbb D^n,$ $\varphi^{-1}(\mathbb D^n) =\tria$ and $\varphi^{-1}(\mathbb T^n) = \mathbb T^n$ (see \eqref{phi inverse}). We now apply Lemma~\ref{silov boundary} to get the desired conclusion. \hfill $\diamondsuit$
\end{remark}

For the $n$-tuple ${\bf r}=(r_1, \ldots r_n)$ of real numbers in $I:=(0,1),$ the domain $\triangle^{\!n}_{{\bf r}}$ is given by 
\beqn
\triangle^{\!n}_{{\bf r}} =
\Big\{ z \in \mathbb C \times \mathbb C^{n-1}_* : |z_1| < r_1|z_2| < \ldots < \Big(\prod_{k=1}^{n-1} r_k\Big) |z_n| < \prod_{k=1}^{n} r_k\Big\}. 
\eeqn
Note that $\triangle^{\!n}_{{\bf r}} \subseteq \tria.$

To prove Proposition~\ref{Boundary values identification}, we need an expression for the norm on $H^2(\tria)$ involving boundary values. 
\begin{proposition}\label{norm2}
 The norm on the Hardy space $H^2(\tria)$ can be expressed as
\beq \label{Norm formula-2} \|f\|^2= \sup_{{\bf r} \in  I^n } \frac{1}{2^n\pi^n}\int_{\partial_{d}(\triangle^{\!n}_{{\bf r}})} |f|^2 d\sigma_{\bf r}, \quad f \in H^2(\tria). 
\eeq
Here $\sigma_{\bf r}$ denotes the surface area measure on $\partial_{d}(\triangle^{\!n}_{{\bf r}})$ induced by the Lebesgue measure on $\mathbb T^n.$   
\end{proposition}
\begin{proof}
    By Lemma~\ref{silov boundary}, the distinguished boundary $\partial_{d}(\triangle^{\!n}_{{\bf r}})$ of $\triangle^{\!n}_{{\bf r}}$ is given by
\beqn
\partial_{d}(\triangle^{\!n}_{{\bf r}}) = \{(z_1, \ldots, z_n)\in \mathbb{C}^n : |z_j|=\prod_{k=j}^n r_k,~ j=1, \ldots, n\}.
\eeqn
It is now easy to see that norms given in \eqref{Norm formula-1} and \eqref{Norm formula-2} coincide.
\end{proof}
Let $\mathcal{I}$ be the set given by $\{\alpha \in \mathbb Z^n : \sum_{j=1}^k \alpha_j +k-1 \Ge 0,~ k=1, \ldots, n\}.$ From \eqref{Norm formula-2}, It turns out that an orthonormal basis for $H^2(\tria)$ is $\{z^{\alpha} :\alpha \in \mathcal{I}\}.$ Furthermore, each function $f \in H^2(\tria)$ has the Laurent series representation as  $\displaystyle f(z)=\sum_{\alpha \in \mathcal{I}}a_{\alpha}z^{\alpha}$ for some scalars $a_{\alpha}.$

 Let $\zeta=(\zeta_1, \ldots, \zeta_n) \in \mathbb T^n$ and let $H^2(\partial_{d}(\tria))$ denote the closed subspace of $L^2(\partial_{d}(\tria))$ defined as \beqn H^2(\partial_{d}(\tria)) := \left \{ f(\zeta)=\sum_{\alpha \in \mathcal {I}}a_{\alpha}\zeta^{\alpha}:\|\{a_{\alpha}\}\|_{\ell ^2}<\infty \right \}.
\eeqn
Let $f \in H^2(\tria).$ We associate a function $\Tilde{f} \in L^2(\partial_{d}(\tria))$ to $f$ defined as
\beq \label{isometric isomorphism}
\tilde{f}(\zeta):=\sum_{\alpha \in \mathcal {I}}a_{\alpha}\zeta^{\alpha},\quad \zeta \in \mathbb T^n.
\eeq
For ${\bf r} \in I^n,$ set
\beqn
f_{\bf r}(\zeta_1, \ldots, \zeta_n):=f\left(\prod_{k=1}^n r_k \zeta_1, \prod_{k=2}^n r_k \zeta_2 \ldots, r_n \zeta_n\right).
\eeqn
Since
\beqn
\|f-f_{\bf r}\|^2_{_{L^2(\partial_{d}(\tria))}}= \sum_{\alpha \in \mathcal{I}}|a_{\alpha}|^2 \left(1 - \prod_{k=1}^n r_k^{\sum_{j=1}^k \alpha_j}\right)^2, 
\eeqn
by an application of the dominated convergence theorem, we have
\beqn
\|f-f_{\bf r}\|^2_{_{L^2(\partial_{d}(\tria))}} \rightarrow 0 \mbox{~as~} r_1, \ldots, r_n \rightarrow 1.
\eeqn
Hence, the function $\Tilde{f}$ is the boundary value function for $f.$ On the other hand, any function in the space $H^2(\partial_{d}(\tria))$ can be represented in the form given by \eqref{isometric isomorphism}. Consequently, it can be extended to become a function in the space $H^2(\tria).$ The above discussion yields the following.
\begin{proposition}  \label{Boundary values identification}
The Hardy space $H^2(\tria)$ is isometrically isomorphic to $H^2(\partial_{d}(\tria)).$
\end{proposition}
An application of Proposition~\ref{Boundary values identification} shows that a function $$f(z)=\sum_{\alpha \in \mathbb Z^n}a_{\alpha} \varphi(z)^{\alpha} \in L^2({\partial_{d}(\tria)})$$ is in $H^2(\tria)$ if and only if $a_{\alpha}=0$ whenever $\alpha \notin \mathbb Z^n_+,$ where $\varphi$ is as given in \eqref{phi-map}. 

Since the space $H^{\infty}(\tria)$ is contained in $H^2(\tria),$ with the aid of Proposition \ref{Boundary values identification}, we now define inner functions on $\tria.$
\begin{definition} \label{def-inner}
   A function $\theta \in H^{\infty}(\tria)$ is called {\it inner} if $\theta$ is unimodular on $\mathbb{T}^n.$
\end{definition} 
\begin{remark}
    If the boundary function of a function $f \in H^2(\mathbb D)$ is essentially bounded then $f$ is bounded on $\mathbb D$ (see \cite[Corollary~1.1.24]{MR2007}). Interestingly, this is not true for $H^2(\tria).$ In fact, the function $f(z)=\frac{1}{z_2 \cdots z_n} \in H^2(\tria)$ is unbounded, however, the associated boundary function is essentially bounded. \hfill $\diamondsuit$
\end{remark}
As the distinguished boundary of $\mathbb D^n$ and $\tria$ are the same (see Remark \ref{Distinguishd boundary same}), to avoid any confusion, we denote the Hardy space on $\mathbb D^n$ by the notation $H^{2}(\partial_{d}(\mathbb D^n)).$

\section{Toeplitz operators on $n$-dimensional Hartogs triangle} \label{3}
Let $L^{\infty}(\mathbb T^n)$ denote the standard $C^*$-algebra consisting of complex-valued Lebesgue measurable functions on ${\mathbb T}^n,$ which are essentially bounded. Let $\Omega=\mathbb D^n \mbox{~or~} \tria.$ Let $P_{ H^{2}(\partial_{d}(\Omega))}$ denote the orthogonal projection from $L^2(\mathbb T^n)$ onto $H^{2}(\partial_{d}(\Omega))$ (see Remark~\ref{Hardy space polydisc} and Proposition~\ref{Boundary values identification}).

For $\phi \in L^{\infty}(\mathbb T^n),$ the {\it Toeplitz operator} $T_{\phi, \Omega}$ corresponding to symbol $\phi$ is defined by
$$T_{\phi, \Omega}(f) = P_{ H^{2}(\partial_{d}(\Omega))}(\phi f),\,\, f\in H^{2}(\partial_{d}(\Omega)).$$
The operator $H_{\phi, \Omega} : H^2(\partial_{d}(\Omega)) \rightarrow L^2(\mathbb T^n) \ominus H^2(\partial_{d}(\Omega))$ defined by \beqn H_{\phi, \Omega}(f)=(I-P_{ H^{2}(\partial_{d}(\Omega))})(\phi f), \quad f \in H^2(\partial_{d}(\Omega)) \eeqn is called a {\it Hankel operator}. It is easy to see that the map $\Psi : L^2(\mathbb T^n) \rar L^2(\mathbb T^n)$ defined by \beq \label{Psi map} \Psi(f)=J_{\varphi^{-1}} \cdot f\circ \varphi^{-1}, \quad f \in L^2(\mathbb T^n)  \eeq
is a unitary map. Note that $\Psi^{-1} : L^2(\mathbb T^n) \rar L^2(\mathbb T^n) $ is given by 
\beq \label{Psi inverse}
\Psi^{-1}(f)= J_{\varphi} \cdot f \circ \varphi, \quad f \in L^2(\mathbb T^n).
\eeq
In view of Remark \ref{Hardy space polydisc} and Proposition \ref{Boundary values identification}, it is worth noting that the restriction map $\Psi : H^{2}(\partial_{d}(\tria)) \rightarrow H^{2}(\partial_{d}(\mathbb D^n))$ given by
\beq \label{Psi map-new} \Psi(f)=J_{\varphi^{-1}} \cdot f\circ \varphi^{-1}, \quad f \in H^{2}(\partial_{d}(\tria))
\eeq
is again a unitary map.
\begin{remark} \label{IFR}
 Note that $H^{\infty}(\mathbb D^n) \subsetneq H^{\infty}(\tria).$ However, it is worth noting that $\Psi^{-1}(H^{\infty}(\mathbb D^n)) \neq H^{\infty}(\tria).$ Indeed, $\Psi^{-1}(1)= J_{\varphi}$ which is not bounded on $\tria,$ where $J_{\varphi}$ is as given in \eqref{Jacobian phi}. On the other hand, an application of Riemann removable singularity theorem (see \cite[Theorem~4.2.1]{S2005}) yields that $H^{\infty}(\mathbb D^n) = H^{\infty}(\mathbb D \times \mathbb D_*^{n-1})$ and consequently, $\phi \in H^{\infty}(\tria)$ if and only if 
$\phi \circ \varphi^{-1} \in H^{\infty}(\mathbb D^n).$ \hfill $\diamondsuit$
\end{remark} Here, we record some algebraic properties of the Toeplitz operators on the Hardy space $H^2(\tria).$ Whenever the context is clear, we will use $P$ in place of $P_{ H^{2}(\partial_{d}(\tria))}.$
\begin{proposition} \label{Injective}
    Let $\phi, \psi \in L^{\infty}(\mathbb T^n).$ Then 
    \begin{itemize}
        \item [$\mathrm{(i)}$] $T_{\phi, \tria}^*$ = $T_{\overline{\phi}, \tria}.$
        \item [$\mathrm{(ii)}$] $T_{\phi, \tria} T_{\psi, \tria}= T_{\phi \psi, \tria}$ if $\overline{\phi}$ or $\psi$ is analytic.
        \item [$\mathrm{(iii)}$] $T_{\phi, \tria} T_{\psi, \tria} -T_{\phi \psi, \tria}=-H^{*}_{\overline{\phi}, \tria} H_{\psi, \tria}.$
        \item [$\mathrm{(iv)}$] If $T_{\phi, \tria}$ is the zero operator on $H^2(\tria)$ then $\phi =0$ a.e.
    \end{itemize}
\end{proposition}
\begin{proof}
    $\mathrm{(i)}$ For $f,g \in H^2(\tria),$ we have
    \allowdisplaybreaks
    \beqn
    \langle T_{\phi, \tria}^* f, g \rangle = \langle f, P(\phi g) \rangle = \langle f, \phi g \rangle  &=&\langle \overline{\phi}f, g \rangle = \langle P(\overline{\phi}f), g \rangle = \langle T_{\overline{\phi}, \tria} f, g \rangle.
    \eeqn
    Hence, $T_{\phi, \tria}^*$ = $T_{\overline{\phi}, \tria}.$
    
 $\mathrm{(ii)}$ Assume that $\psi$ is analytic. Then for $f \in H^2(\tria),$ $$T_{\phi, \tria} T_{\psi, \tria}(f)=T_{\phi, \tria}(\psi f)=P(\phi \psi f)= T_{\phi \psi, \tria}.$$
    Similarly, we have the conclusion when $\overline{\phi}$ is analytic.
    
$\mathrm{(iii)}$ For $f, g \in H^2(\tria),$
\allowdisplaybreaks
\beqn
\langle H_{\overline{\phi}, \tria}^* H_{\psi, \tria}f, g \rangle &=& \langle (I-P)(\psi f), (I-P)(\overline{\phi} g) \rangle \\ &=& \langle \psi f , \overline{\phi} g  \rangle -  \langle \psi f, P(\overline{\phi} g) \rangle \\ &=& \langle P (\phi \psi f) , g  \rangle -  \langle P(\psi f), P(\overline{\phi} g) \rangle
\\ &=& \langle (-T_{\phi, \tria} T_{\psi, \tria} + T_{\phi \psi, \tria})(f), g \rangle.
\eeqn
Hence, $T_{\phi, \tria} T_{\psi, \tria} -T_{\phi \psi, \tria}=-H_{\overline{\phi}, \tria}^* H_{\psi, \tria}.$

$\mathrm{(iv)}$ Let $\phi(z)= \displaystyle \sum_{\gamma \in \mathbb Z^n}a_{\gamma}z^\gamma.$ For $\alpha, \beta \in \mathcal {I},$ we have 
    \beqn
      0 = \langle T_{\phi, \tria}z^{\alpha}, z^{\beta} \rangle  = \langle P(\phi z^{\alpha}), z^{\beta} \rangle 
      = \langle\phi z^{\alpha}, z^{\beta} \rangle 
      &=& \langle \sum_{\gamma \in \mathbb Z^n}a_{\gamma}z^{\gamma +\alpha}, z^{\beta} \rangle \\
      &=& a_{\beta - \alpha}.
    \eeqn
Note that $\{\beta -\alpha : \alpha, \beta \in \mathcal {I} \} =\mathbb Z^n.$ Hence, $a_{\gamma} =0$ for all $\gamma \in \mathbb Z^n.$ This completes the proof.
\end{proof}
\begin{remark} In view of Proposition \ref{Injective}$\mathrm{(iv)},$ the map $\phi \rightarrow T_{\phi, \tria}$ from $L^{\infty}(\mathbb T^n)$ into the set of all Toeplitz operators on $\tria$ is injective. \hfill $\diamondsuit$
\end{remark}
\section{The main theorem}\label{4}
The following result illustrates the relationship between Toeplitz operators on the n-dimensional Hartogs triangle and the unit polydisc. It shows that every Toeplitz operator on $H^2(\tria)$ is unitarily equivalent to some Toeplitz operator on $ H^2(\mathbb D^n)$ and vice versa. 
Note that $\phi \in L^{\infty}(\mathbb T^n)$ if and only if $\phi \circ \varphi^{-1} \in L^{\infty}(\mathbb T^n)$ if and only if $\phi \circ \varphi\in L^{\infty}(\mathbb T^n).$ 
\begin{theorem} \label{Toeplitz- relation}
For $\phi \in L^{\infty}(\mathbb T^n),$ $$ T_{\phi, \tria} = \Psi^{-1}T_{\phi \circ \varphi^{-1}, \mathbb D^n}\Psi \mbox{~and~} T_{\phi, \mathbb D^n} = \Psi T_{\phi \circ \varphi, \tria}\Psi^{-1} $$ where $\Psi$ is as given in \eqref{Psi map-new}.
\end{theorem}
\begin{proof}
Let $f\in L^2(\mathbb T^n).$ Then $ \displaystyle f(z)=\sum_{\alpha \in \mathbb Z^n}a_{\alpha}z^{\alpha}.$ Now,
\allowdisplaybreaks
\begin{align*}
    \Psi P_{ H^{2}(\partial_{d}(\tria))}\Psi^{-1}(f)(z)& \overset{\eqref{Psi inverse}}=\Psi P_{H^{2}(\partial_{d}(\tria))}(J_{\varphi} \cdot f \circ \varphi)(z) \\ &= \Psi P_{H^{2}(\partial_{d}(\tria))}\left(\sum_{\alpha \in \mathbb Z^n}a_{\alpha }J_{\varphi}(z)\varphi(z)^{\alpha} \right)\\ &=\Psi\left(\sum_{\alpha \in \mathbb Z^n_+}a_{\alpha}J_{\varphi}(z)\varphi(z)^{\alpha}\right) \\ & \overset{\eqref{Psi map}}=J_{\varphi^{-1}}(z)\left(\sum_{\alpha \in \mathbb Z^n_+}a_{\alpha}J_{\varphi}(\varphi^{-1}(z))(\varphi(\varphi^{-1}(z))^{\alpha}\right) \\&= \sum_{\alpha \in \mathbb Z^n_+}a_{\alpha}z^{\alpha} \\ &= P_{ H^{2}(\partial_{d}(\mathbb D^n))}(f)(z).
\end{align*}
Hence, \beq \label{Projections relation} \Psi P_{H^{2}(\partial_{d}(\tria))}\Psi^{-1}=P_{H^{2}(\partial_{d}(\mathbb D^n))}. \eeq
The following diagram summarizes the observation.
\begin{center}
\begin{tikzcd}
 L^2(\mathbb T^n)\arrow[r, "\Psi"] \arrow[d, , "P_{H^{2}(\partial_{d}(\tria))}" ]
& L^2(\mathbb T^n) \arrow[d, "P_{ H^{2}(\partial_{d}(\mathbb D^n))}" ] \\
H^{2}(\partial_{d}(\tria)) \arrow[r, "\Psi" ]
&  H^{2}(\partial_{d}(\mathbb D^n)).
\end{tikzcd}
\end{center}
Let $\phi \in L^{\infty}(\mathbb T^n).$ For $h \in H^2(\partial_{d}(\mathbb D^n)),$
\begin{align*}
  \Psi T_{\phi, \tria} \Psi^{-1}(h)(z) &= \Psi P_{H^{2}(\partial_{d}(\tria))}(\phi  \cdot \Psi^{-1}(h))(z) \\ &\overset{\eqref{Psi map-new}}= \Psi P_{H^{2}(\partial_{d}(\tria))}\Big( J_{\varphi}(z) \cdot \phi(z) \cdot h \circ \varphi(z)\Big) \\ & \overset{\eqref{Psi map-new}}= \Psi P_{H^{2}(\partial_{d}(\tria))}\Psi^{-1}(\phi\circ \varphi^{-1} \cdot h)(z) \\ &\overset{\eqref{Projections relation}}= P_{H^{2}(\partial_{d}(\mathbb D^n))}(\phi\circ \varphi^{-1} \cdot h)(z) \\ &= T_{\phi \circ \varphi^{-1}, \mathbb D^n} (h)(z)
\end{align*}
Hence, \beq \label{10} \Psi T_{\phi, \tria} \Psi^{-1}=T_{\phi \circ \varphi^{-1}, \mathbb D^n}. \eeq
By replacing $\phi$ by $\phi \circ \varphi$ in \eqref{10}, we get the remaining part. 
\end{proof}
Let $j=1, \ldots, n$ and let $\mathscr M_{z_j}$ denote the linear operator of multiplication by the coordinate function $z_j$ in $H^2(\tria)$:
\beqn
\mathscr M_{z_j}f=z_j f~\mbox{whenever}~f \in H^2(\tria)~\mbox{and}~{z_j} f \in H^2(\tria).
\eeqn 
By \cite[Proposition 4.2]{CJP2023}, $\mathscr M_{z_j}, j=1, \ldots, n,$ is a bounded linear operator on $H^2(\tria).$ Also, let $M_j, j=1, \ldots,n,$ denote the bounded linear operator of multiplication by the coordinate function $z_j$ in $H^2(\mathbb D^n)$. 

Here are some immediate consequences of the preceding theorem. 
Compact operators on $H^2(\mathbb D^n)$ are characterized in terms of multiplication operators $M_1, \ldots M_n$ on $H^2(\mathbb D^n).$ In fact, a bounded linear map $T$ on $H^2(\mathbb D^n)$ is compact if and only if $M_j^{*k} T M_{i}^{k} \rightarrow 0 $ in norm for all $1 \Le i,j \Le n$ (see \cite[Theorem 3.2]{MSS2018}).  However, for $\tria,$ we obtain only the necessary conditions for the compactness of bounded linear operators on $H^2(\tria).$
\begin{corollary}\label{compact}
   If a bounded linear map $T$ on $H^2(\tria)$ is compact then $\mathscr M_{z_j}^{*k}T\mathscr M_{z_i}^{k} \rightarrow 0 $ in norm as $k \rightarrow \infty$ for all $1 \Le i,j \Le n.$ 
\end{corollary}
\begin{proof}
    Let T be a compact operator on $H^2(\tria).$ Then $\Psi T \Psi^{-1}$ is compact on $H^2(\mathbb D^n).$ Now for $1 \Le i,j \Le n$ and $k \in \mathbb N,$ by \cite[Proposition 8.3]{CJP2023},
    \begin{equation*}
        \mathscr M_{z_j}^{*k}T\mathscr M_{z_i}^{k}= \left(\Psi^{-1} \prod_{l=j}^nM_{l} \Psi \right)^{*k}T \left(\Psi^{-1} \prod_{p=i}^nM_{p} \Psi \right)^k. 
    \end{equation*}
  which is equivalent to  
    \begin{equation*}
        \mathscr M_{z_j}^{*k}T\mathscr M_{z_i}^{k}= \Psi^{-1} \prod_{l=j}^nM^{*k}_{l} \Psi T \Psi^{-1} \prod_{p=i}^nM^k_{p} \Psi .
    \end{equation*}
    Hence,
    \begin{equation} \label{Relation M-M}
       \| \mathscr M_{z_j}^{*k}T\mathscr M_{z_i}^{k}\| \Le \| M^{*k}_{n} \Psi T \Psi^{-1} M^k_{n} \|,
    \end{equation}
    which goes to $0$ as $k \rightarrow \infty$ (see \cite[Theorem 3.2]{MSS2018}).
 \end{proof}
 The converse of the above corollary may not be true. 
\begin{example}
    Consider the bounded linear operator $T:=I-\mathscr M_{z_n}\mathscr M_{z_n}^*$ on $H^2(\tria),$ where $I$ is the identity operator on $H^2(\tria).$ An application of \cite[Eqn~(4.11)]{CJP2023} yields that \beqn \mathscr M_{z_j}^{*k}T\mathscr M_{z_i}^{k}=0, \quad k \in \mathbb N, ~1 \Le i,j \Le n.\eeqn However, it is easy to see from \cite[Lemma 6.2]{CJP2023} that the operator $T$ is noncompact. \eof
\end{example} 
Recall that a bounded linear operator $T$ on a Hilbert space $H$ is said to be a {\it shift} if $T$ is an isometry and $T^{*k} \rightarrow 0$ as $k \rightarrow 0$ in the strong operator topology. 

\begin{corollary}
If $\phi$ is a nonconstant inner function on $\tria,$ then $T_{\phi, \tria}$ is a shift. 
\end{corollary}
\begin{proof}
Let $\phi \in H^{\infty}(\tria)$ be a nonconstant inner function. By Remark \ref{IFR}, $\phi \circ \varphi^{-1}$ is inner function on $\mathbb D^n.$ It is easy to see that $T_{\phi \circ
\varphi^{-1}, \mathbb D^n}$ is an isometry. This combined with Theorem~\ref{Toeplitz- relation} yields that $T_{\phi, \tria}$ is an isometry. Following the von Neumann-Wold decomposition for isometries, it suffices to show that 
$$
\displaystyle \bigcap_{k=0}^{\infty} \phi^{k} H^2(\tria)=\{0\}.
$$
Now,
\beqn
\displaystyle \bigcap_{k=0}^{\infty} \phi^{k} H^2(\tria)& \overset{\eqref{Psi inverse}}=& \displaystyle \bigcap_{k=0}^{\infty} \phi^{k} \Psi^{-1} H^2(\mathbb D^n) \\ &=& \displaystyle \bigcap_{k=0}^{\infty} \Psi^{-1}\left ((\phi \circ \varphi^{-1})^{k}  H^2(\mathbb D^n)\right) \\ &=& \displaystyle \Psi^{-1}\left (\bigcap_{k=0}^{\infty} (\phi \circ \varphi^{-1})^{k}  H^2(\mathbb D^n)\right) = \{0\},
\eeqn
where the last equality follows from \cite[Theorem 4.1]{KDPS2022}. \end{proof}

\section{Applications of the main theorem} \label{5}
In this section, we present several applications of Theorem~\ref{Toeplitz- relation}. We begin with the following theorem, an analogue of the Brown-Halmos theorem for $n$-dimensional Hartogs triangle. Recall that the Brown-Halmos theorem states that a bounded linear operator $T$ on $H^2(\mathbb D)$ is Toeplitz if and only if $M^*_zTM_z=T,$ where $M_z$ is the multiplication operator by the coordinate function $z$ in $H^2(\mathbb D)$ (see \cite[Theorem 6]{BH1964}). 

\begin{theorem}\label{Brown-Halmos}
    A bounded linear operator $T$ on $H^2(\tria)$ is a Toeplitz operator if and only if $\mathscr M^*_{z_j}T \mathscr M_{z_j}=T$ for all $j = 1, \ldots, n.$
\end{theorem}
\begin{proof}
Consider a Toeplitz operator $T_{\phi, \tria}$ on $H^2(\tria)$ with the symbol $\phi \in L^{\infty}(\mathbb T^n).$ By Theorem~\ref{Toeplitz- relation}, we have $\Psi T_{\phi, \tria} \Psi^{-1}=T_{\phi \circ \varphi^{-1}, \mathbb D^n}.$
From \cite[Theorem 3.1]{MSS2018}, we obtain $M_j^*\Psi T_{\phi, \tria} \Psi^{-1}M_j=\Psi T_{\phi, \tria} \Psi^{-1}$ for all $j=1, \ldots, n.$ Consequently,
\beq \label{relation-1}
\prod_{k=j}^nM_k^*\Psi T_{\phi, \tria} \Psi^{-1}\prod_{k=j}^nM_k= \Psi T_{\phi, \tria}\Psi^{-1}, \quad j=1, \ldots, n.
\eeq

Furthermore, as per \cite[Proposition 8.3]{CJP2023}, we have \beqn \mathscr M^*_{z_j}T_{\phi, \tria} \mathscr M_{z_j}=T_{\phi, \tria}, \quad j = 1, \ldots, n.\eeqn

Conversely, let $T$ be a bounded linear operator on $H^2(\tria)$ such that $\mathscr M^*_{z_j}T \mathscr M_{z_j}=T$ for all $j = 1, \ldots, n.$ By \cite[Proposition 8.3]{CJP2023}, we have 
\beq
\label{relation-2} 
\prod_{k=j}^nM_j^*\Psi T \Psi^{-1}\prod_{k=j}^nM_j= \Psi T \Psi^{-1}, \quad j=1, \ldots, n.
\eeq
Thus, for $j=n,$ $M_n^*\Psi T \Psi^{-1}M_n= \Psi T \Psi^{-1}.$
This combined with \eqref{relation-2} (with $j=n-1$) yields  
\beqn
\Psi T \Psi^{-1}=M_{n-1}^*M_n^*\Psi T \Psi^{-1}M_nM_{n-1}=M_{n-1}^*\Psi T \Psi^{-1}M_{n-1}.  
\eeqn
One may now proceed by a finite induction to conclude that $M_j^*\Psi T \Psi^{-1}M_j= \Psi T \Psi^{-1}$ for all $j=1, \ldots, n.$
Thus, by \cite[Theorem 3.1]{MSS2018}, $\Psi T \Psi^{-1}$ is a Toeplitz operator on $H^2(\mathbb D^n).$ Hence, by Theorem~\ref{Toeplitz- relation}, $T$ is a Toeplitz operator on $H^2(\tria).$
\end{proof}
Similar to the unit disc case, we get an analogous result for $\tria.$ 
\begin{corollary}
    The only compact Toeplitz operator on $H^2(\tria)$ is the zero operator.
\end{corollary}  
\begin{proof}
Let $T_{\phi, \tria}$ be a compact Toeplitz operator on $H^2(\tria)$ with the symbol $\phi \in L^{\infty}(\mathbb T^n).$ By Theorem~\ref{Toeplitz- relation}, $T_{\phi \circ \varphi^{-1}, \mathbb D^n}$ is compact on $H^2(\mathbb D^n).$
    Since $\mathscr M^*_{z_1}T_{\phi, \tria} \mathscr M_{z_1}=T_{\phi, \tria}$ (see Theorem \ref{Brown-Halmos}), by Theorem~\ref{Toeplitz- relation} and \cite[Proposition 8.3]{CJP2023}, \beqn \prod_{j=1}^n M_{j}^*T_{\phi \circ \varphi^{-1}, \mathbb D^n}\prod_{j=1}^n M_j = T_{\phi \circ \varphi^{-1}, \mathbb D^n}.\eeqn As a consequence, for any $\alpha, \beta \in \mathbb Z^n_+,$ and $r \in \mathbb Z_+,$
    \beq \label{Compact zero}
    |\langle T_{\phi \circ \varphi^{-1}, \mathbb D^n}z^{\alpha}, z^{\beta} \rangle|=|\langle T_{\phi \circ \varphi^{-1}, \mathbb D^n}z^{\alpha+r\sum_{j=1}^n\varepsilon_j}, z^{\beta+r\sum_{j=1}^n\varepsilon_j} \rangle|.\eeq
Compact operators map weakly convergent sequences to norm convergent sequences (see \cite[Proposition~VI.3.3]{C1990}) and $\{z^{\alpha+r\sum_{j=1}^n\varepsilon_j}\}$ converges to $0$ weakly as $r \rightarrow \infty.$ Therefore, for any $\alpha \in \mathbb Z^n_+,$ $\|z^{\alpha+r\sum_{j=1}^n\varepsilon_j}\| \rightarrow 0$ as $r \rightarrow \infty.$  This combined with an application of Cauchy-Schwarz inequality in \eqref{Compact zero} yields that for all $\alpha, \beta \in \mathbb Z^n_+,$ $\langle T_{\phi \circ \varphi^{-1}, \mathbb D^n}z^{\alpha}, z^{\beta} \rangle=0.$ Hence, $T_{\phi, 
\tria}=0.$ This combined with Proposition \ref{Injective}$\mathrm{(iv)}$ yields that $\phi=0.$
\end{proof}
\begin{remark}
    For a function $\phi \in L^{\infty}(\mathbb T^n),$ it is easy to see from Theorem~\ref{Toeplitz- relation} that  $H_{\phi, \tria} = \Psi^{-1}H_{\phi \circ \varphi^{-1}, \mathbb D^n}\Psi,$ where $\Psi$ is as given in \eqref{Psi map-new}. Combining this with \cite[Theorem~5]{AYZ2009} yields that the zero operator is the only compact Hankel operator on $H^2(\tria).$ \hfill $\diamondsuit$
\end{remark}

Recall that a bounded linear operator T on a Hilbert space $H$ is called {\it partial isometry} if $T|_{ker(T)^{\perp}}$ is an isometry. The following theorem characterizes partial isometric Toeplitz operators on $H^2(\tria).$ Set $\Tilde{z}_j:=\frac{z_{j}}{z_{j+1}}$ for $j=1, \ldots, n-1,$ and $\Tilde{z}_n:=z_n.$
\begin{theorem} \label{Partial isometry result}
    For a nonzero function $\phi \in L^{\infty}(\mathbb T^n),$ the Toeplitz operator $T_{\phi, \tria}$ is a partial isometry if
and only if there exist inner functions $\theta_1, \theta_2 \in H^{\infty}(\tria)$ such that $\theta_1$ and $\theta_2$ depend on the different variables among $\Tilde{z_1}, \ldots, \Tilde{z}_n$ and $$T_{\phi, \tria}=T_{\theta_1, \tria}^*T_{\theta_2, \tria}.$$
\end{theorem}
\begin{proof}
 Let $T_{\phi, \tria}$ be a partial isometry on $H^2(\tria).$ Then $\Psi T_{\phi, \tria}\Psi^{-1}$ is a partial isometry on $H^2(\mathbb D^n).$ It follows from \cite[Theorem 1.1]{KDPS2022} that there exist inner functions $\phi_1,\phi_2 \in H^{\infty}(\mathbb D^n)$ such that 
 $$\Psi T_{\phi, \tria}\Psi^{-1}=T_{\phi_1, \mathbb D^n}^*T_{\phi_2, \mathbb D^n}.$$
Consequently,
 \beq \label{Partial poly} T_{\phi, \tria}=\Psi^{-1}T_{\phi_1, \mathbb D^n}^*\Psi \Psi^{-1} T_{\phi_2, \mathbb D^n}\Psi.\eeq
Set $\theta_1=\phi_1 \circ \varphi$ and $\theta_2=\phi_2 \circ \varphi.$ Since $\Psi^{-1}(H^2(\mathbb D^n))=H^2(\tria),$ $J_{\varphi} \cdot \phi_1 \circ \varphi$ and  $J_{\varphi} \cdot \phi_2 \circ \varphi$ are unimodular on $\mathbb T^n.$ Combining Remark \ref{IFR} with the fact that $\mathscr M_{z_j}$ is a bounded linear operator on $H^2(\tria)$ for every $j=2, \ldots, n,$ we conclude that both $\theta_1$ and $\theta_2$ are inner functions on $\tria.$
Now, by Theorem~\ref{Toeplitz- relation} and \eqref{Partial poly}, we have 
\beqn T_{\phi, \tria}= T_{\phi_1 \circ \varphi, \tria}^*  T_{\phi_2 \circ \varphi, \tria}=T_{\theta_1, \tria}^*T_{\theta_2, \tria}.\eeqn
Since $\phi_1$ and $\phi_2$ depend on separate variables among $z_1, \ldots, z_n$ (see \cite[Theorem 1.1]{KDPS2022}), it follows that $\theta_1$ and $\theta_2$ depend on different variables among $\Tilde{z_1}, \ldots, \Tilde{z}_n.$

Conversely, let there exist inner functions $\theta_1, \theta_2 \in H^{\infty}(\tria)$ such that $$T_{\phi, \tria}=T_{\theta_1, \tria} ^*T_{\theta_2, \tria},$$
which is equivalent to
$$\Psi T_{\phi, \tria}\Psi^{-1}=\Psi T^*_{\theta_1, \tria}\Psi^{-1}\Psi T_{\theta_2, \tria}\Psi^{-1}.$$
Now, by Theorem~\ref{Toeplitz- relation},
$$T_{\phi \circ \varphi^{-1}, \mathbb D^n}=T_{\theta_1 \circ \varphi^{-1}, \mathbb D^n}^*T_{\theta_2 \circ \varphi^{-1}, \mathbb D^n}.$$
By Remark \ref{IFR}, $\theta_1 \circ \varphi^{-1}$ and $\theta_2 \circ \varphi^{-1}$ are inner functions on $\mathbb D^n.$ Since $\theta_1$ and $\theta_2$ depend on different variables among $\Tilde{z_1}, \ldots, \Tilde{z}_n,$ we get that $\theta_1 \circ \varphi^{-1}$ and $\theta_2 \circ \varphi^{-1}$ depend on different variables among $z_1, \ldots, z_n.$ Hence, by \cite[Theorem 1.1]{KDPS2022}, $T_{\phi, \tria}$ is a partial isometry. \end{proof}
A Toeplitz operator $T_{\phi, \tria}$ on $H^2(\tria)$ is called an {\it analytic Toeplitz operator} if $\phi \in H^{\infty}(\tria).$ And it is called a {\it co-analytic Toeplitz operator} if $T^*_{\phi, \tria}$ is an analytic Toeplitz operator.

In the following proposition, we collect some basic properties of analytic Toeplitz operators on $H^2(\tria).$   
\begin{proposition}
    Let $T_{\phi, \tria}$ be a Toeplitz operator on $H^2(\tria)$ with the symbol $\phi \in L^{\infty}(\mathbb T^n).$ Then the following are equivalent $:$
    \begin{itemize}
        \item [$\mathrm{(i)}$] $T_{\phi, \tria}$ is an analytic Toeplitz operator.
        \item [$\mathrm{(ii)}$] For $j=1, \ldots, n,$ $T_{\phi, \tria}$ commutes with $\mathscr{M}_{z_j}.$
        \item [$\mathrm{(iii)}$]  For $j=1, \ldots, n,$ the range of $\mathscr{M}_{z_j}$ is invariant under $T_{\phi, \tria}.$
    \end{itemize}    
    
\end{proposition}
\begin{proof}The equivalence of statements $\mathrm{(i)}$ and $\mathrm{(ii)}$ follows from \cite[Theorem 6.1]{CJP2023}.

$\mathrm{(ii)} \iff \mathrm{(iii)}:$ Fix $j=1, \ldots, n.$ Let $T_{\phi, \tria}$ commute with $\mathscr{M}_{z_j},$ that is, \beqn T_{\phi, \tria}\mathscr{M}_{z_j}=\mathscr{M}_{z_j}T_{\phi, \tria}.\
\eeqn 
Now for any $f \in H^2(\tria),$ $T_{\phi, \tria}\mathscr{M}_{z_j}(f)=\mathscr{M}_{z_j}T_{\phi, \tria}(f)$ which is equivalent to  $T_{\phi, \tria}(\operatorname{ran}(\mathscr{M}_{z_j}) \subseteq \operatorname{ran}(\mathscr{M}_{z_j}).$ Hence, $\mathrm{(iii)}$ follows. 

Conversely, let the range of $\mathscr{M}_{z_j}$ is invariant under $T_{\phi, \tria}$ for each $j=1, \ldots, n.$ Fix $j=1, \ldots, n.$ Since $\operatorname{ran}(\mathscr{M}_{z_j})$ is closed, for any $f \in H^2(\tria),$ there exists $h_{f} \in  H^2(\tria)$ such that 
\beqn
T_{\phi, \tria}\mathscr{M}_{z_j}(f)=\mathscr{M}_{z_j}(h_f),
\eeqn
which is equivalent to 
\beqn
\mathscr{M}^*_{z_j}T_{\phi, \tria}\mathscr{M}_{z_j}(f)=h_f,
\eeqn
by Theorem~\ref{Brown-Halmos}, we have $T_{\phi, \tria}(f)=h_f.$ Hence, $T_{\phi, \tria}\mathscr{M}_{z_j}=\mathscr{M}_{z_j}T_{\phi, \tria}.$ This completes the proof. 
\end{proof}

Here is an immediate corollary of the above proposition.

\begin{corollary}
Let $T_{\phi, \tria}$ be a Toeplitz operator on $H^2(\tria)$ with the symbol $\phi \in L^{\infty}(\mathbb T^n).$ Then the following are equivalent $:$

 \begin{itemize}
       \item [$\mathrm{(i)}$] $T_{\phi, \tria}$ is a co-analytic Toeplitz operator.
       \item [$\mathrm{(ii)}$] For $j=1, \ldots, n,$ $T_{\phi, \tria}$ commutes with $\mathscr{M}^*_{z_j}.$
       \item [$\mathrm{(iii)}$]  For $j=1, \ldots, n,$ the range of $\mathscr{M}_{z_j}$ is invariant under $T^*_{\phi, \tria}.$
\end{itemize}
 \end{corollary}

In the following result, we characterize left invertible analytic Toeplitz operators on $H^2(\tria).$ 
\begin{proposition} \label{Left invertible}
    Let $\phi \in H^{\infty}(\tria).$ Then
    \begin{itemize}
    \item [$\mathrm{(i)}$] Either $\operatorname{ran}(T_{\phi, \tria})$ is dense in $H^2(\tria)$ or $\operatorname{ker}(T_{\phi, \tria}^*)$ is not finite dimensional.

    \item [$\mathrm{(ii)}$] $T_{\phi, \tria}$ is left invertible if and only if $\frac{1}{\phi} \in L^{\infty}(\mathbb T^n).$

        \item [$\mathrm{(iii)}$] $T_{\phi, \tria}$ is left invertible and $\operatorname{ker}(T_{\phi, \tria}^*)$ is not finite dimensional if and only if $\frac{1}{\phi} \in L^{\infty}(\mathbb T^n)$ but $\frac{1}{\phi} \notin H^{\infty}(\tria).$ 
    \end{itemize}
\end{proposition}
\begin{proof}
    $\mathrm{(i)}$ The proof follows from \cite[Theorem 3]{AC1970}.
    
    $\mathrm{(ii)}$ Let $T_{\phi, \tria}$ be left invertible. Then by Theorem~\ref{Toeplitz- relation}, $T_{\phi \circ \varphi^{-1}, \mathbb D^n}$ is left invertible and hence, $(\phi \circ \varphi^{-1})H^2(\mathbb D^n)$ is closed and invariant under $M_i$ for each $i=1, \ldots, n.$ Consequently, an application of \cite[Theorem 2]{KS2017} yields that $\phi \circ \varphi^{-1}$ is invertible in $L^{\infty}(\mathbb T^n).$ Since $\varphi^{-1}$ maps $\mathbb T^n$ onto $\mathbb T^n$ (see Remark \ref{Distinguishd boundary same}), $\phi$ is invertible in $L^{\infty}(\mathbb T^n).$

    Conversely, let $\frac{1}{\phi} \in L^{\infty}(\mathbb T^n).$ Then for $f \in H^2(\tria),$
    $$T_{\frac{1}{\phi}, \tria} T_{\phi, \tria}(f)=T_{\frac{1}{\phi}, \tria}(\phi f)=f.$$
    Hence, $T_{\phi, \tria}$ is left invertible.

    $\mathrm{(iii)}$ Note that $\frac{1}{\phi} \in H^{\infty}(\tria)$ if and only if $\phi H^2(\tria) =H^2(\tria).$ This combined with $\mathrm{(i)}$ and $\mathrm{(ii)}$ yields $\mathrm{(iii)}.$
\end{proof}
We conclude this section with a result contrasting from the classical Hardy space case.
\begin{example}
    The Coburn alternative states that at least one of $T_{\phi, \mathbb D}$ or $T^*_{\phi, \mathbb D}$ is injective for a nonzero function $\phi \in L^{\infty}(\mathbb T)$ (see \cite[Theorem~3.3.10]{MR2007}). In the case of $\tria,$ however, it fails. For instance, $T_{\phi, \tria}(1)=T_{\phi, \tria}^*(1)=0$ for $\phi=\overline{z_1}z^3_2.$ Indeed, the functions $\overline{z_1}z^3_2$ and $z_1\overline{z_2}^3$ are in the orthogonal complement of $H^2(\partial_d(\tria))$ in $L^2(\mathbb T^n).$ It is worth noting that for a nonzero holomorphic or anti-holomorphic symbol $\phi,$ Coburn alternative holds. \eof
\end{example}

{}

\end{document}